\documentclass[a4paper,11pt,leqno]{smfart}

\usepackage{amssymb,amsmath,enumerate,verbatim,mathrsfs,graphics,graphicx,mathptm,float,mathtools}
\usepackage[T1]{fontenc}
\usepackage[english, francais]{babel}
\usepackage[all]{xy}
\usepackage[pdfstartview=FitH,
            colorlinks=true,
            linkcolor=blue,
            urlcolor=blue,
            citecolor=red,
            bookmarks,
            bookmarksopen=true,
            bookmarksnumbered=true]{hyperref}

\usepackage{cleveref}
\theoremstyle{plain}

\newtheorem{theoalph}{Theor\`eme}

\newtheorem{thmalph}[theoalph]{Th\'eor\`eme}

\theoremstyle{definition}

\theoremstyle{remark}

\theoremstyle{plain}
\newtheorem{thmsec}{Th\'eor\`eme}[section]

\newtheorem{pro}[thmsec]{Proposition}
\newtheorem{lem}[thmsec]{Lemme}
\newtheorem{cor}[thmsec]{Corollaire}

\theoremstyle{definition}
\newtheorem{defin}[thmsec]{D\'efinition}

\theoremstyle{remark}
\newtheorem{rem}[thmsec]{Remarque}

\def\og{\leavevmode\raise.3ex\hbox{$\scriptscriptstyle\langle\!\langle$~}}
\def\fg{\leavevmode\raise.3ex\hbox{~$\!\scriptscriptstyle\,\rangle\!\rangle$}}

\addtocounter{section}{0}             
\numberwithin{equation}{section}       

\newcommand{\pp}{\mathbb{P}^{2}_{\mathbb{C}}}
\newcommand{\pd}{\mathbb{\check{P}}^{2}_{\mathbb{C}}}

\newcommand\Sing{\mathrm{Sing}}

\newcommand\Leg{\mathrm{Leg}}
\newcommand\IF{\mathrm{I}_{\mathcal{F}}}
\newcommand\IinvF{\mathrm{I}_{\mathcal{F}}^{\mathrm{inv}}}
\newcommand\ItrF{\mathrm{I}_{\mathcal{F}}^{\hspace{0.2mm}\mathrm{tr}}}

\newcommand\F{\mathcal{F}}
\newcommand\W{\mathcal{W}}

\newcommand\omegaoverline{{\mspace{2mu}\overline{\mspace{-1.4mu}\omega\mspace{-1.4mu}}\mspace{2mu}}}

\newcommand\elltext{\scalebox{1.1}{\ensuremath \ell}}

\begin{document}
\title[Classification des feuilletages convexes de degr\'{e} deux]{Une nouvelle démonstration de la classification des feuilletages convexes de degr\'{e} deux sur $\pp$}
\date{\today}

\author{Samir \textsc{Bedrouni}}

\address{Facult\'e de Math\'ematiques, USTHB, BP $32$, El-Alia, $16111$ Bab-Ezzouar, Alger, Alg\'erie}
\email{sbedrouni@usthb.dz}

\author{David \textsc{Mar\'{\i}n}}

\thanks{D. Mar\'{\i}n acknowledges financial support from the Spanish Ministry of Economy and Competitiveness, through grant MTM2015-66165-P and the "Mar\'{\i}a de Maeztu" Programme for Units of Excellence in R\&D (MDM-2014-0445).}

\address{BGSMath and Departament de Matem\`{a}tiques Universitat Aut\`{o}noma de Barcelona E-08193  Bellaterra (Barcelona) Spain}

\email{davidmp@mat.uab.es}


\maketitle{}

\begin{altabstract}
\selectlanguage{english}
A holomorphic foliation on $\pp,$ or a real analytic foliation on $\mathbb{P}^{2}_{\mathbb{R}},$ is said to be convex if its leaves other than straight lines have no inflection points. The classification of the convex foliations of degree $2$ on $\pp$ has been established in $2015$ by C.~\textsc{Favre} and J.~\textsc{Pereira}. The main argument of this classification was a result obtained in~$2004$  by~D.~\textsc{Schlomiuk} and N.~\textsc{Vulpe} concerning the real polynomial vector fields of degree $2$ whose associated foliation on $\mathbb{P}^{2}_{\mathbb{R}}$ is convex. We present here a new proof of this classification, that is simpler, does not use this result and does not leave the holomorphic framework. It is based on the properties of certain models of convex foliations of $\pp $ of arbitrary degree and of the discriminant of the dual web of a foliation of $\pp$.
\end{altabstract}

\selectlanguage{french}
\begin{abstract}
\noindent Un feuilletage holomorphe sur $\pp$ ou analytique réel sur $\mathbb{P}^{2}_{\mathbb{R}}$ est dit convexe si ses feuilles qui ne sont pas des droites n'ont pas de points d'inflexion. La classification des feuilletages convexes de degré $2$ sur $\pp$ a été établie~en $2015$ par~C.~\textsc{Favre} et J.~\textsc{Pereira}. L'argument principal de cette classification était un résultat obtenu en~$2004$ par~D.~\textsc{Schlomiuk} et N.~\textsc{Vulpe} concernant les champs de vecteurs réels polynomiaux de degré $2$ dont le feuilletage de $\mathbb{P}^{2}_{\mathbb{R}}$ associé est convexe. Nous présentons ici une nouvelle démonstration de cette classification, plus simple, n'utilisant pas ce résultat et ne sortant pas du cadre holomorphe; elle s'appuie sur des propriétés de certains modèles de feuilletages convexes de $\pp$ de degré quelconque et du discriminant du tissu dual d'un feuilletage de $\pp.$
{\it Classification math\'ematique par sujets (2010). --- 37F75, 32S65, 32M25.}
\end{abstract}

\section*{Introduction}
\bigskip

\noindent L'ensemble $\mathbf{F}(d)$ des feuilletages de degré $d$ sur $\pp$ s'identifie à un ouvert de \textsc{Zariski} dans un espace projectif de dimension $(d+2)^{2}-2$ sur lequel agit le groupe~$\mathrm{Aut}(\pp).$ Suivant \cite{MP13} un feuilletage de $\mathbf{F}(d)$ est dit {\sl convexe} si ses feuilles qui ne sont pas des droites n'ont pas de points d'inflexion.

\noindent D'après \cite[Proposition~2, page~23]{Bru15} tout feuilletage de degré~$0$ ou~$1$ est convexe. Pour $d\geq2$, l'ensemble des feuilletages convexes de $\mathbf{F}(d)$ est un fermé de \textsc{Zariski} propre de $\mathbf{F}(d)$ et il contient les feuilletages $\mathcal{H}_{1}^{d}$, resp. $\F_{1}^{d}$, resp. $\F_{0}^{d}$ définis en carte affine par les $1$-formes (\emph{voir} \cite[Proposition~4.1]{BM18}, \cite[page~75]{Bed17} et \cite[page~179]{MP13})
\begin{align*}
&\omega_{\hspace{0.2mm}1}^{\hspace{0.2mm}d}=y^d\mathrm{d}x-x^d\mathrm{d}y,&&
\text{resp.}\hspace{1.5mm}\omegaoverline_{1}^{d}=y^{d}\mathrm{d}x+x^{d}(x\mathrm{d}y-y\mathrm{d}x),&&
\text{resp.}\hspace{1.5mm}\omegaoverline_{0}^{d}=(x^{d}-x)\mathrm{d}y-(y^{d}-y)\mathrm{d}x.
\end{align*}
\noindent Les feuilletages $\mathcal{H}_{1}^{d}$ et $\F_{1}^{d}$ appartiennent tous deux à l'adhérence dans $\mathbf{F}(d)$ de l'orbite sous l'action de $\mathrm{Aut}(\pp)$ du feuilletage $\F_{0}^{d},$ dit feuilletage de \textsc{Fermat} de degré $d.$ Notons de plus que $\mathcal{H}_{1}^{d}$ est {\sl homogène} au sens où il est invariant par homothétie.

\noindent En $2015$ \textsc{Favre} et \textsc{Pereira} \cite[Proposition~7.4]{FP15} ont classifié les feuilletages convexes de degré $2$ sur $\pp.$ Plus précisément ils ont montré le résultat suivant.
\begin{thmalph}[\cite{FP15}]\label{thmalph:class-feuilletage-convexe-2}
{\sl \`A automorphisme de $\pp$ pr\`es, il y a trois feuilletages convexes de degré deux sur le plan projectif complexe, à savoir les feuilletages $\mathcal{H}_{1}^{2}$, $\mathcal{F}_{1}^{2}$ et $\mathcal{F}_{0}^{2}$ d\'ecrits respectivement en carte affine par les~$1$-formes suivantes

\begin{itemize}
\item [\texttt{1. }] $\omega_{1}^{2}=y^2\mathrm{d}x-x^2\mathrm{d}y$;
\smallskip
\item [\texttt{2. }] $\omegaoverline_{1}^{2}=y^{2}\mathrm{d}x+x^{2}(x\mathrm{d}y-y\mathrm{d}x)$;
\smallskip
\item [\texttt{3. }] $\omegaoverline_{0}^{2}=(x^{2}-x)\mathrm{d}y-(y^{2}-y)\mathrm{d}x$.
\end{itemize}
}
\end{thmalph}

\noindent L'argument fondamental de cette classification était le résultat de \textsc{Schlomiuk} et \textsc{Vulpe} dans \cite[Théorème 50]{SV04}. Ces derniers donnent en carte affine une liste \cite[Table~2]{SV04} de formes normales pour les feuilletages convexes de degré $2$ sur $\mathbb{P}^{2}_{\mathbb{R}}$ (ils considèrent en effet des champs de vecteurs de type $\mathrm{X}=A(x,y)\frac{\partial{}}{\partial{x}}+B(x,y)\frac{\partial{}}{\partial{y}},$ où $A$ et $B$ sont des polynômes de $\mathbb{R}[x,y]$ vérifiant $\mathrm{pgcd}(A,B)=1$ et $\max(\deg A,\deg B)=2$). \textsc{Favre} et \textsc{Pereira} \cite[Proposition~7.4]{FP15} ont remarqué que les arguments de \cite[Théorème~50]{SV04} s'appliquent de façon identique aux champs de vecteurs complexes. Leur démonstration a ainsi consisté à réduire le nombre de modèles de champs de vecteurs réels présentés dans \cite[Table~2]{SV04} en cherchant ceux qui sont conjugués par un automorphisme de~$\pp.$

\noindent Dans cet article nous donnons une nouvelle démonstration de cette classification, n'utilisant pas \cite[Théorème~50]{SV04} et ne sortant pas du cadre analytique complexe; elle repose sur certaines propriétés des feuilletages $\mathcal{H}_{1}^{d},$ $\F_{0}^{d}$, $\F_{1}^{d}$ (\cite[Proposition~4.1]{BM18}, \cite[Proposition~6.3]{BM18}, Proposition~\ref{pro:F1^d} démontrée au~\S\ref{sec:Demonstration-Theoreme-A}) et du discriminant du tissu dual d'un feuilletage de $\pp$ (\cite[Lemme~2.2]{BFM13} et \cite[Proposition~3.3]{MP13}), \emph{voir} \S\ref{sec:Demonstration-Theoreme-A}.


\section{Singularités, diviseur d'inflexion et tissu dual d'un feuilletage de $\pp$}\label{sec:singularité-diviseur-inflexion-tissu-dual}
\bigskip

\noindent Un feuilletage holomorphe $\mathcal{F}$ de degré $d$ sur~$\pp$ est défini en coordonnées homogènes $[x:y:z]$ par une $1$-forme du type  $$\omega=a(x,y,z)\mathrm{d}x+b(x,y,z)\mathrm{d}y+c(x,y,z)\mathrm{d}z,$$ o\`{u} $a,$ $b$ et $c$ sont des polynômes homogènes de degré $d+1$ sans facteur commun satisfaisant la condition d'\textsc{Euler} $i_{\mathrm{R}}\omega=0$, où $\mathrm{R}=x\frac{\partial{}}{\partial{x}}+y\frac{\partial{}}{\partial{y}}+z\frac{\partial{}}{\partial{z}}$ désigne le champ radial et $i_{\mathrm{R}}$ le produit intérieur par $\mathrm{R}$. Le {\sl lieu singulier} $\mathrm{Sing}\mathcal{F}$ de $\mathcal{F}$ est le projectivisé du lieu singulier de~$\omega$ $$\mathrm{Sing}\omega=\{(x,y,z)\in\mathbb{C}^3\,\vert \, a(x,y,z)=b(x,y,z)=c(x,y,z)=0\}.$$

\noindent Rappelons quelques notions locales attachées au couple $(\mathcal{F},s)$, où $s\in\Sing\mathcal{F}$. Le germe de $\F$ en $s$ est défini, à multiplication près par une unité de l'anneau local $\mathcal{O}_s$ en $s$, par un champ de vecteurs
\begin{small}
$\mathrm{X}=A(\mathrm{u},\mathrm{v})\frac{\partial{}}{\partial{\mathrm{u}}}+B(\mathrm{u},\mathrm{v})\frac{\partial{}}{\partial{\mathrm{v}}}$.
\end{small}
\noindent La {\sl multiplicité algébrique} $\nu(\mathcal{F},s)$ de $\mathcal{F}$ en $s$ est donnée par $$\nu(\mathcal{F},s)=\min\{\nu(A,s),\nu(B,s)\},$$ où $\nu(g,s)$ désigne la multiplicité algébrique de la fonction $g$ en $s$. L'{\sl ordre de tangence} entre $\mathcal{F}$ et une droite générique passant par $s$ est l'entier $$\hspace{0.8cm}\tau(\mathcal{F},s)=\min\{k\geq\nu(\mathcal{F},s)\hspace{1mm}\colon\det(J^{k}_{s}\,\mathrm{X},\mathrm{R}_{s})\neq0\},$$ où $J^{k}_{s}\,\mathrm{X}$ est le $k$-jet de $\mathrm{X}$ en $s$ et $\mathrm{R}_{s}$ est le champ radial centré en $s$.
\smallskip

\noindent La singularité $s$ de $\mathcal{F}$ est dite {\sl radiale} si $\nu(\mathcal{F},s)=1$ et si de plus $\tau(\mathcal{F},s)\geq2$. Si tel est le cas, l'entier naturel $\tau(\mathcal{F},s)-1$, compris entre $1$ et $d-1,$ est appelé l'{\sl ordre de radialité} de $s.$
\medskip

\noindent Rappelons la notion du diviseur d'inflexion de $\F$. Soit $\mathrm{Z}=E\frac{\partial}{\partial x}+F\frac{\partial}{\partial y}+G\frac{\partial}{\partial z}$ un champ de vecteurs homogène de degré $d$ sur $\mathbb{C}^3$ non colinéaire au champ radial décrivant $\mathcal{F},$ {\it i.e.} tel que $\omega=i_{\mathrm{R}}i_{\mathrm{Z}}\mathrm{d}x\wedge\mathrm{d}y\wedge\mathrm{d}z.$ Le {\sl diviseur d'inflexion} de $\mathcal{F}$, noté $\IF$, est le diviseur défini par l'équation
\begin{equation}\label{equa:ext1}
\left| \begin{array}{ccc}
x &  E &  \mathrm{Z}(E) \\
y &  F &  \mathrm{Z}(F)  \\
z &  G &  \mathrm{Z}(G)
\end{array} \right|=0.
\end{equation}
Ce diviseur a été étudié dans \cite{Per01} dans un contexte plus général. En particulier, les propriétés suivantes ont été prouvées.
\begin{enumerate}
\item [\texttt{1.}] Sur $\mathbb{P}^{2}_{\mathbb{C}}\smallsetminus\mathrm{Sing}\mathcal{F},$ $\IF$ coïncide avec la courbe décrite par les points d'inflexion des feuilles de $\mathcal{F}$;
\item [\texttt{2.}] Si $\mathcal{C}$ est une courbe algébrique irréductible invariante par $\mathcal{F},$ alors $\mathcal{C}\subset \IF$ si et seulement si $\mathcal{C}$ est une droite invariante;
\item [\texttt{3.}] $\IF$ peut se décomposer en $\IF=\IinvF+\ItrF,$ où le support de $\IinvF$ est constitué de l'ensemble des droites invariantes par $\mathcal{F}$ et où le support de $\ItrF$ est l'adhérence des points d'inflexion qui sont isolés le long des feuilles de $\mathcal{F}$;
\item [\texttt{4.}] Le degré du diviseur $\IF$ est $3d.$
\end{enumerate}

\begin{defin}[\cite{MP13}]
\noindent Un feuilletage $\mathcal{F}$ sur $\pp$ est dit {\sl convexe} si son diviseur d'inflexion $\IF$ est totalement invariant par $\mathcal{F}$, {\it i.e.} si $\IF$ est le produit de droites invariantes par $\F.$
\end{defin}
\medskip

\noindent Rappelons maintenant la définition d'un $k$-tissu sur une surface complexe $S$.
\begin{defin}
Soit $k\geq1$ un entier. Un {\sl $k$-tissu (global)} $\mathcal{W}$ sur $S$ est la donnée d'un recouvrement ouvert $(U_{i})_{i\in I}$ de $S$ et d'une collection de $k$-formes symétriques $\omega_{i}\in \mathrm{Sym}^{k}\Omega^{1}_{S}(U_{i})$, à zéros isolés, satisfaisant:
\begin{itemize}
\item [($\mathfrak{a}$)] il existe $g_{ij}\in \mathcal{O}^{*}_{S}(U_{i}\cap U_{j})$ tel que $\omega_i$ coïncide avec $g_{ij}\omega_{j}$ sur $U_{i}\cap U_{j}$;
\item [($\mathfrak{b}$)] en tout point générique $m$ de $U_{i},$ $\omega_{i}(m)$ se factorise en produit de $k$ formes linéaires deux à deux non colinéaires.
\end{itemize}
\end{defin}
\vspace{-2mm}

\noindent Le {\sl discriminant} $\Delta(\mathcal{W})$ de $\mathcal{W}$ est le diviseur défini localement par $\Delta(\omega_i)=0$, où $\Delta(\omega_i)$ est le discriminant de la $k$-forme symétrique~$\omega_{i}\in\mathrm{Sym}^{k}\Omega^{1}_{S}(U_{i})$, \emph{voir} \cite[Chapitre~1, \S 1.3.4]{PP15}. Le support de $\Delta(\mathcal{W})$ est constitué des points de~$S$ qui ne vérifient pas la condition ($\mathfrak{b}$). Lorsque~$k=1$ cette condition est toujours vérifiée et on retrouve la définition usuelle d'un feuilletage holomorphe~$\mathcal{F}$ sur $S.$

\noindent Un $k$-tissu global $\W$ sur $S$ est dit {\sl décomposable} s'il existe des tissus globaux $\mathcal{W}_{1},\mathcal{W}_{2}$ sur $S$ n'ayant pas de sous-tissus communs tels que $\mathcal{W}$ soit la superposition de $\mathcal{W}_{1}$ et $\mathcal{W}_{2}$; on écrira $\mathcal{W}=\mathcal{W}_{1}\boxtimes\mathcal{W}_{2}.$ Dans le cas contraire $\W$ est dit {\sl irréductible}. On dit que $\mathcal{W}$ est {\sl complètement décomposable} s'il existe des feuilletages globaux $\mathcal{F}_{1},\ldots,\mathcal{F}_{k}$ sur $S$ tels que $\mathcal{W}=\mathcal{F}_{1}\boxtimes\cdots\boxtimes\mathcal{F}_{k}.$ Pour plus de détails sur ce sujet, nous renvoyons à \cite{PP15}.
\smallskip

\noindent Revenons au cas qui nous intéresse: $S=\pp$. Se donner un $k$-tissu sur $\pp$ revient à se donner une $k$-forme symétrique polynomiale $\omega=\sum_{i+j=k}a_{ij}(x,y)\mathrm{d}x^{i}\mathrm{d}y^{j}$, à zéros isolés et de discriminant non identiquement nul. Ainsi tout $k$-tissu sur $\pp$ peut se lire dans une carte affine donnée $(x,y)$ de $\pp$ par une équation différentielle polynomiale $F(x,y,y')=0$ de degré $k$ en $y'$.

\noindent Suivant~\cite{MP13} à tout feuilletage $\F$ de degré $d\geq1$ sur $\pp$ est associé un $d$-tissu irréductible sur le plan projectif dual $\pd$, appelé {\sl transformée de \textsc{Legendre}} (ou {\sl tissu dual}) de $\F$, et noté $\Leg\F$; les feuilles de $\Leg\F$ sont les droites tangentes aux feuilles de $\F.$ Plus explicitement, soit $(x,y)$ une carte affine de $\pp$ et considérons la carte affine $(p,q)$ de $\pd$ associée à la droite $\{y=px-q\}\subset{\mathbb{P}^{2}_{\mathbb{C}}}$; si $\F$ est défini par une $1$-forme $\omega=A(x,y)\mathrm{d}x+B(x,y)\mathrm{d}y,$ où $A,B\in\mathbb{C}[x,y],$ $\mathrm{pgcd}(A,B)=1$, alors $\Leg\F$ est donné par l'équation différentielle implicite
\[
\check{F}(p,q,x):=A(x,px-q)+pB(x,px-q)=0, \qquad \text{avec} \qquad x=\frac{\mathrm{d}q}{\mathrm{d}p}.
\]
L'application de Gauss est l'application rationnelle $\mathcal{G}_{\F}\hspace{1mm}\colon\pp\dashrightarrow \pd$ qui à un point régulier $m$ associe la droite tangente $\mathrm{T}_{\hspace{-0.1mm}m}\F$. Si $\mathcal{C}\subset\pp$ est une courbe passant par certains points singuliers de $\F$, on définit $\mathcal{G}_{\mathcal{F}}(\mathcal{C})$ comme étant l'adhérence de $\mathcal{G}_{\F}(\mathcal{C}\setminus\Sing\F)$. Il résulte de \cite[Lemme~2.2]{BFM13} que
\begin{equation}\label{equa:Delta-LegF}
\Delta(\Leg\F)=\mathcal{G}_{\F}(\ItrF)\cup\check{\Sigma}_{\F},
\end{equation}
où $\check{\Sigma}_{\F}$ désigne l'ensemble des droites duales des points de $\Sigma_{\F}:=\{s\in\Sing\F\hspace{0.8mm}:\hspace{0.8mm}\tau(\F,s)\geq2\}$.

\section{Démonstration du Théorème~\ref{thmalph:class-feuilletage-convexe-2}}\label{sec:Demonstration-Theoreme-A}
\bigskip

\noindent Comme nous l'avons dit dans l'Introduction, la preuve du Théorème~\ref{thmalph:class-feuilletage-convexe-2} s'appuie sur certaines propriétés des feuilletages  $\mathcal{H}_{1}^{d}$, $\F_{0}^{d},$ $\F_{1}^{d}$ et du discriminant du tissu dual d'un feuilletage de $\pp.$

\noindent Notons d'abord que le feuilletage homogène $\mathcal{H}_{1}^{d}$ possède deux singularités radiales d'ordre maximal $d-1.$ En fait cette propriété caractérise la $\mathrm{Aut}(\pp)$-orbite de $\mathcal{H}_{1}^{d}$ comme le montre le résultat suivant, déduit~de~\cite[Proposition~4.1]{BM18}:
\begin{pro}[\cite{BM18}]\label{pro:H1^d}
{\sl Soit $\mathcal{H}$ un feuilletage homogène de degré $d\geq2$ sur $\pp$ ayant deux singularités radiales distinctes d'ordre maximal $d-1.$ Alors $\mathcal{H}$ est linéairement conjugué au feuilletage $\mathcal{H}_{1}^{d}$ défini par la~$1$-forme~$\omega_{\hspace{0.2mm}1}^{\hspace{0.2mm}d}=y^d\mathrm{d}x-x^d\mathrm{d}y.$
}
\end{pro}

\noindent On en tire en particulier le:
\begin{cor}\label{cor:class-feuilletage-homogène-convexe-2}
{\sl \`A automorphisme de $\pp$ pr\`es, il y a un et un seul feuilletage homogène convexe de degré~$2$ sur le plan projectif complexe, à savoir le feuilletage $\mathcal{H}_{1}^{2}$ décrit par la $1$-forme~$\omega_{1}^{2}=y^2\mathrm{d}x-x^2\mathrm{d}y.$
 }
\end{cor}

\begin{proof}[\sl D\'emonstration]
On sait d'après~\cite[Proposition~2.2]{BM18} que tout feuilletage homogène convexe de degré supérieur ou égal à $2$ sur $\pp$ possède au moins deux singularités radiales distinctes. L'énoncé découle alors immédiatement de \cite[Proposition~4.1]{BM18} (\emph{cf.} Proposition~\ref{pro:H1^d} ci-dessus) et de la remarque évidente suivante: si~un feuilletage $\F$ de degré $2$ sur~$\pp$ admet une singularité radiale $s$, alors l'ordre de radialité $\tau(\F,s)-1$ de~$s$ est égal à $1.$
\end{proof}

\noindent Une propriété caractéristique de la $\mathrm{Aut}(\pp)$-orbite du feuilletage de \textsc{Fermat} $\F_{0}^{d}$ est donnée par \cite[Proposition~6.3]{BM18}:
\begin{pro}[\cite{BM18}]\label{pro:F0-d}
{\sl Soit $\F$ un feuilletage de degré $d\geq2$ sur $\pp$ ayant trois singularités radiales d'ordre maximal $d-1$, non alignées. Alors $\F$ est linéairement conjugué au feuilletage de \textsc{Fermat} $\F_{0}^{d}$ défini par la $1$-forme~$\omegaoverline_{0}^{d}=(x^{d}-x)\mathrm{d}y-(y^{d}-y)\mathrm{d}x.$
}
\end{pro}

\noindent Une propriété caractéristique de la $\mathrm{Aut}(\pp)$-orbite du feuilletage $\F_{1}^{d}$ est donnée par la:
\begin{pro}\label{pro:F1^d}
{\sl Soit $\F$ un feuilletage convexe de degré $d\geq2$ sur $\pp.$ Supposons que $\F$ possède une singularité $s_1$ de multiplicité algébrique maximale $d$ et une singularité $s_2$ radiale d'ordre maximal $d-1.$ Alors
\begin{itemize}
\item [--] ou bien $\F$ est homogène;
\item [--] ou bien $\F$ est linéairement conjugué au feuilletage $\F_{1}^{d}$ décrit par la $1$-forme~$\omegaoverline_{1}^{d}=y^{d}\mathrm{d}x+x^{d}(x\mathrm{d}y-y\mathrm{d}x).$
\end{itemize}
}
\end{pro}

\begin{rem}\label{rem:caractérisation-feuilletage-homogène}
Notons qu'un feuilletage de degré $d$ sur $\pp$ est homogène si et seulement s'il possède une singularité de multiplicité algébrique maximale $d$ et une droite invariante ne passant pas par cette singularité.
\end{rem}

\begin{proof}
Choisissons un système de coordonnées homogènes $[x:y:z]\in\pp$ tel que $s_1=[0:0:1]$ et $s_2=[0:1:0].$ Par hypothèse nous avons $\nu(\F,s_1)=d$, $\nu(\F,s_2)=1$ et $\tau(\F,s_2)=d.$ L'égalité $\nu(\F,s_1)=d$ assure que toute $1$-forme $\omega$ décrivant $\F$ dans la carte affine $z=1$ est du type
\begin{align*}
\omega=A_{d}(x,y)\mathrm{d}x+B_{d}(x,y)\mathrm{d}y+C_d(x,y)(x\mathrm{d}y-y\mathrm{d}x),
\end{align*}
où $A_{d}$, $B_{d}$ et $C_{d}$ sont des polynômes homogènes de degré $d.$
\vspace{1mm}

\noindent Dans la carte affine $y=1$ le feuilletage $\F$ est donné par
\begin{align*}
\theta=-C_{d}(x,1)\mathrm{d}x-B_{d}(x,1)\mathrm{d}z+A_{d}(x,1)(z\mathrm{d}x-x\mathrm{d}z)\hspace{0.5mm};
\end{align*}
nous avons $\theta\wedge(z\mathrm{d}x-x\mathrm{d}z)=Q(x,z)\mathrm{d}x\wedge\mathrm{d}z$, où $Q(x,z)=x\hspace{0.1mm}C_{d}(x,1)+z\hspace{0.1mm}B_{d}(x,1).$ L'égalité $\tau(\F,s_2)=d$ se traduit alors par le fait que le polynôme $Q\in\mathbb{C}[x,z]$ est homogène non nul de degré $d+1$, ce qui permet d'écrire $B_d(x,y)=\beta\hspace{0.1mm}x^d$,\, $C_d(x,y)=\delta x^d,$\, avec $\beta,\delta\in\mathbb{C},$ $|\beta|+|\delta|\neq0.$ Par suite nous avons $J^{1}_{(0,0)}\theta=A_{d}(0,1)(z\mathrm{d}x-x\mathrm{d}z)$; alors l'égalité $\nu(\F,s_2)=1$ assure que $A_{d}(0,1)\neq0.$ Ainsi $\theta$ s'écrit
\begin{align}\label{equa:theta}
&\hspace{2.5cm}\theta=-x^d\mathrm{d}\left(\delta x+\beta z\right)+A_{d}(x,1)\left(z\mathrm{d}x-x\mathrm{d}z\right),
&&\quad\beta,\delta\in\mathbb{C},\hspace{1mm}|\beta|+|\delta|\neq0,\hspace{1mm}A_{d}(0,1)\neq0.
\end{align}

\noindent Supposons que $\F$ ne soit pas homogène; comme $\nu(\F,s_1)=d$, toute droite invariante par $\F$ doit passer par~$s_1$ (Remarque~\ref{rem:caractérisation-feuilletage-homogène}) et doit donc être de la forme $(ax+by=0)$, $[a:b]\in\mathbb{P}^{1}_{\mathbb{C}}.$ Or nous remarquons à partir de~(\ref{equa:theta})~que la droite $\elltext=(\delta x+\beta z=0)$ est invariante par $\F.$ Il en résulte que $\beta=0,$ $\delta\neq0$ et~$\elltext=(x=0).$

\noindent Par conséquent
\begin{align*}
&\hspace{2.5cm}\omega=A_{d}(x,y)\mathrm{d}x+\delta\hspace{0.1mm}x^d(x\mathrm{d}y-y\mathrm{d}x),
&& \delta\,A_{d}(0,1)\in\mathbb{C}^{*}.
\end{align*}

\noindent Posons $P(t)=A_{d}(1,t)$; puisque $A_{d}(0,1)\neq0$ le polynôme $P\in\mathbb{C}[t]$ est de degré $d.$ De plus l'homogénéité de~$A_{d}$ entraîne que
\begin{align*}
&\omega=x^d\left(P\left(\frac{y}{x}\right)\mathrm{d}x+\delta(x\mathrm{d}y-y\mathrm{d}x)\right),
\qquad\hspace{3mm} \delta\in\mathbb{C}^{*}.
\end{align*}

\noindent Nous constatons qu'une droite de la forme $(y=rx)$, $r\in\mathbb{C},$ est invariante par $\F$ si et seulement si $r$ est une racine du polynôme $P.$ Il s'en suit que l'ensemble des droites invariantes par $\F$ est constitué de la droite $(x=0)$ et des droites $(y=rx),$ où  $r$ parcourt l'ensemble des racines de $P.$
\vspace{1mm}

\noindent En coordonnées homogènes, le feuilletage $\F$ est décrit par le champ de vecteurs $0\frac{\partial}{\partial x}+A_{d}(x,y)\frac{\partial}{\partial y}+\delta\hspace{0.1mm}x^d\frac{\partial}{\partial z}$; d'après la formule (\ref{equa:ext1}), le diviseur d'inflexion $\IF$ de $\mathcal{F}$ est donné par
\begin{equation*}
0=\left| \begin{array}{ccc}
x  & 0                       & 0                                    \\
y  & A_{d}(x,y)              & \dfrac{\partial A_{d}}{\partial y}(x,y)\\
z  & \delta\hspace{0.1mm}x^d & 0
\end{array} \right|
=-\delta\hspace{0.1mm}x^{d+1}A_{d}(x,y)\dfrac{\partial A_{d}}{\partial y}(x,y)=-\delta\hspace{0.1mm}x^{3d}P\left(\frac{y}{x}\right)P'\left(\frac{y}{x}\right).
\end{equation*}
Comme $\F$ est par hypothèse convexe, nous en déduisons que toute racine de $P'$ est une racine de $P$ et donc que $P$ est divisible par sa dérivée $P'.$ Le polynôme $P$ est alors nécessairement de la forme $P(t)=\alpha(t-a)^d,$ avec $a\in\mathbb{C}$ et $\alpha\in\mathbb{C}^{*},$ d'où
\begin{align*}
&\hspace{1.5cm}\omega=\alpha(y-ax)^{d}\mathrm{d}x+\delta\hspace{0.1mm}x^d(x\mathrm{d}y-y\mathrm{d}x),
\qquad\hspace{3mm} a\in\mathbb{C},\hspace{1mm} \alpha,\hspace{0.2mm}\delta\in\mathbb{C}^{*}.
\end{align*}

\noindent La $1$-forme $\omega$ est linéairement conjuguée à $\omegaoverline_{1}^{d}=y^{d}\mathrm{d}x+x^{d}(x\mathrm{d}y-y\mathrm{d}x)$; en effet
\[
\hspace{-0.9cm}\omegaoverline_{1}^{d}=\frac{\raisebox{-0.5mm}{$\delta^{d+1}$}}{\alpha^{d+2}}\varphi^*\omega,
\quad \text{où}\hspace{1.5mm}
\varphi=\left(\frac{\raisebox{-0.5mm}{$\alpha$}}{\delta}\hspace{0.1mm}x,\frac{\raisebox{-0.5mm}{$\alpha$}}{\delta}\left(y+ax\right)\right).
\]
\end{proof}

\noindent Le lemme suivant joue un rôle crucial dans la preuve du Théorème~\ref{thmalph:class-feuilletage-convexe-2}.
\begin{lem}\label{lem:cardinal(SigmaF)}
{\sl Soit $\F$ un feuilletage convexe de degré $d\geq2$ sur $\pp.$ Alors l'ensemble $\Sigma_{\F}$ des points singuliers~$s\in\Sing\F$ tels que $\tau(\F,s)\geq2$ est de cardinal supérieur ou égal à~$2.$
}
\end{lem}

\begin{proof}[\sl D\'emonstration]
En vertu de la convexité de $\F,$ la formule~(\ref{equa:Delta-LegF}) entraîne que le discriminant $\Delta(\Leg\F)$ de $\Leg\F$ est constitué des droites duales des points de $\Sigma_{\F}.$ Si $\Sigma_{\F}$ était vide ou réduit à un point, il en résulterait que $\pp\setminus\Delta(\Leg\F)$ serait simplement connexe de sorte que (\emph{cf.} \cite[Proposition~1.3.1]{PP15}) le $d$-tissu $\Leg\F$ serait complètement décomposable, ce qui contredirait son irréductibilité.
\end{proof}

\begin{rem}\label{rem:SigmaF}
Si $\F$ est un feuilletage de degré $d\geq2$ sur $\pp$, alors $$\Sigma_{\F}=\Sigma_{\F}^{\mathrm{rad}}\cup\{s\in\Sing\F\hspace{0.8mm}:\hspace{0.8mm}\nu(\F,s)\geq2\},$$
où $\Sigma_{\F}^{\mathrm{rad}}$ désigne l'ensemble des singularités radiales de $\F.$
\end{rem}

\begin{rem}\label{rem:mult(Delta(LegF))}
Soit $\F$ un feuilletage de degré $d\geq2$ sur $\pp$ tel que $\Sigma_{\F}^{\mathrm{rad}}$ ne soit pas vide. Soit $s\in\Sigma_{\F}^{\mathrm{rad}}.$ Désignons par $\mathrm{mult}(\Delta(\Leg\F),\check{\ell}_{s})$ la multiplicité du discriminant $\Delta(\Leg\F)$ de $\Leg\F$ le long de la droite $\check{\ell}_{s}$ duale de $s.$ Alors (\emph{voir} \cite[Proposition~3.3]{MP13})
\begin{align*}
\mathrm{mult}(\Delta(\Leg\F),\check{\ell}_{s})\geq\tau(\F,s)-1\hspace{0.5mm};
\end{align*}
l'égalité est réalisée si (condition suffisante) l'ordre de radialité $\tau(\F,s)-1$ de $s$ appartient à $\{d-2,d-1\}.$

\noindent En particulier, si $d=2$, resp. $d=3$, alors, pour tout $s\in\Sigma_{\F}^{\mathrm{rad}},$ on a
\begin{align*}
&\hspace{1cm}\mathrm{mult}(\Delta(\Leg\F),\check{\ell}_{s})=1,
&&
\text{resp.}\hspace{1.5mm}\mathrm{mult}(\Delta(\Leg\F),\check{\ell}_{s})=\tau(\F,s)-1\in\{1,2\}.
\end{align*}
\end{rem}

\begin{proof}[\sl D\'emonstration du Théorème~\ref{thmalph:class-feuilletage-convexe-2}]
Soit $\F$ un feuilletage convexe de degré $2$ sur $\pp.$ Pour tout $s\in\Sing\F$, nous avons $\nu(\F,s)\leq\deg\F=2$. Il résulte alors de la Remarque~\ref{rem:SigmaF} que
$$\Sigma_{\F}=\Sigma_{\F}^{\mathrm{rad}}\cup\{s\in\Sing\F\hspace{0.8mm}:\hspace{0.8mm}\nu(\F,s)=2\}.$$

\noindent Notons (\emph{voir} \cite[Lemme~2.11]{BM19F}) que tout feuilletage de degré supérieur ou égal à $2$ sur $\pp$ a au plus une singularité de multiplicité algébrique maximale, égale à son degré. Nous en déduisons l'alternative suivante:
\begin{itemize}
\item [$\bullet$] ou bien $\Sigma_{\F}=\Sigma_{\F}^{\mathrm{rad}}$;
\item [$\bullet$] ou bien $\Sigma_{\F}=\Sigma_{\F}^{\mathrm{rad}}\cup\{s_{1}\}$ pour un certain $s_1\in\Sing\F$ vérifiant $\nu(\F,s_{1})=2$.
\end{itemize}
Supposons dans un premier temps que $\Sigma_{\F}=\Sigma_{\F}^{\mathrm{rad}}.$ La convexité de $\F$ et la formule~(\ref{equa:Delta-LegF}) impliquent alors que le discriminant $\Delta(\Leg\F)$ du $2$-tissu $\Leg\F$ se décompose en produit de droites duales des points de $\Sigma_{\F}^{\mathrm{rad}}.$ Comme la droite duale de tout point de $\Sigma_{\F}^{\mathrm{rad}}$ est de  multiplicité $1$ dans $\Delta(\Leg\F)$ (Remarque~\ref{rem:mult(Delta(LegF))}) et comme $\deg(\Delta(\Leg\F))=(\deg\F-1)(\deg\F+2)=4$ (\emph{cf.} \cite[page~177]{MP13}), il en résulte que $\#\hspace{0.5mm}\Sigma_{\F}^{\mathrm{rad}}=4.$
Les~quatre~points~de~$\Sigma_{\F}^{\mathrm{rad}}$ ne sont pas alignés, car toute droite de $\pp$ ne peut contenir plus de $\deg\F+1=3$ points singuliers de $\F.$ Il s'en suit en particulier que $\F$ possède trois singularités radiales non-alignées; ceci~entraîne, d'après \cite[Proposition~6.3]{BM18}, que $\F$ est linéairement conjugué au feuilletage de \textsc{Fermat} $\F_{0}^{2}$.
\smallskip

\noindent Supposons maintenant qu'il existe $s_1\in\Sing\F$ tel que $\Sigma_{\F}=\Sigma_{\F}^{\mathrm{rad}}\cup\{s_{1}\}$ et $\nu(\F,s_{1})=2.$ Par le Lemme~\ref{lem:cardinal(SigmaF)}, $\#\hspace{0.5mm}\Sigma_{\F}\geq2$ et donc $\Sigma_{\F}^{\mathrm{rad}}$ est non vide; d'où la présence d'une singularité radiale $s_2$ de $\F.$ Par suite, d'après la Proposition~\ref{pro:F1^d}, ou bien $\F$ est homogène, auquel cas $\F$ est linéairement conjugué au feuilletage~$\mathcal{H}_{1}^{2}$ en vertu du Corollaire~\ref{cor:class-feuilletage-homogène-convexe-2}, ou bien $\F$ est linéairement conjugué au feuilletage~$\F_{1}^{2}.$
\end{proof}


\end{document}